\documentclass[12pt]{amsart}
\usepackage{amscd}
\usepackage{ amssymb}
\usepackage[margin=2.7 cm]{geometry}

\newtheorem{theorem}{Theorem}[section]
\newtheorem{lemma}[theorem]{Lemma}
\theoremstyle{definition}

\theoremstyle{remark}

\numberwithin{equation}{section} \theoremstyle{plain}

\def\F{\mathbb F}

\def\B{\mathbb B}
\def\G{{\mathbb G}L}
\def\U{\mathbb U}
\def\V{\mathbb V}
\def\W{\mathbb W}

\def\F{\mathbb F}

\def\B{\mathfrak B}

\def\U{\mathbb U}
\def\V{\mathbb V}
\def\W{\mathbb W}
\def\S{\mathbb S}

\copyrightinfo{2001}{enter name of copyright holder}

\newcommand{\thmref}[1]{Theorem~\ref{#1}}
\newcommand{\lemref}[1]{Lemma~\ref{#1}}

\newcommand{\eqnref}[1]{~{\textrm(\ref{#1})}}

\begin{document}

\title[orthogonal and symplectic groups over algebraically closed fields]{On the Conjugacy Classes in the orthogonal and symplectic groups over algebraically closed fields}
\author{Krishnendu Gongopadhyay}
\address{Theoretical Statistics and Mathematics Unit,
Indian Statistical Institute, 203 B. T. Road,
Kolkata 700108, India}
\email{krishnendug@gmail.com}
\date{November 2, 2009}
\copyrightinfo{2001}{enter name of copyright holder}

\footnote{{\it Mathematics Subject Classification(2000). \hspace{.1in}}{Primary 20G15, 20E45; Secondary 17B10 }}

\begin{abstract}
Let $\F$ be an algebraically closed field. Let $\V$ be a vector space equipped with a non-degenerate symmetric or symplectic bilinear form $B$ over $\F$. Suppose the characteristic of $\F$ is \emph{sufficiently large}, i.e. either zero or greater than the dimension of $\V$. Let $I(\V, B)$ denote the group of isometries. Using the Jacobson-Morozov lemma  we give a new and simple proof of the fact that two elements in $I(\V,B)$ are conjugate if and only if they have the same elementary divisors.
\end{abstract}
\maketitle

\section{Introduction}
Let $\F$ be an algebraically closed field. Let $\V$ be a vector space of dimension $n+1$ over $\F$. Suppose the characteristic of $\F$ is \emph{sufficiently large}, i.e. char$(\F)$ is either zero or greater than the dimension of $\V$. Let $B$ be a non-degenerate symmetric, resp. symplectic (i.e. skew-symmetric),  bilinear form on $\V$. Such a $(\V, B)$ is called a \emph{non-degenerate space}.  Let $I(\V,B)$ denote the group of isometries of $(\V,B)$. It is a linear algebraic group. When $B$ is symmetric, resp. symplectic, $I(\V, B)$ is called the orthogonal, resp. symplectic group of $(\V, B)$.  An element of $I(\V, B)$ will be  called an \emph{isometry}. Let $\W$ be a subspace of $\V$. The restriction of $B$ on $\W$, viz. the form $B: \W \times \W \to \F$, will be denoted by $B|_{\W}$.

By a remarkable property of a linear algebraic group, every isometry of $(\V, B)$ has the unique Jordan decomposition cf. Humphreys \cite{hum}. That is, every isometry $T: \V \to \V$ has the unique decomposition $T=T_s T_u$, where $T_s: \V \to \V$ is semisimple (i.e. every $T_s$-invariant subspace has a $T_s$-invariant complement), $T_u: \V \to \V$ is unipotent (i.e. all eigenvalues are $1$). Moreover,  $T_s$, $T_u$ are elements of $I(\V, B)$, they are polynomials in $T$, and $T_sT_u=T_uT_s$.

When $B$ is symmetric assume $n \geq 2$, and when $B$ is symplectic assume $n \geq 1$.   In these cases $I(\V,B)$ has unipotent isometries. Moreover, the group $I(\V, B)$ is a semisimple algebraic group. Let $T:\V \to \V$ be a unipotent isometry. Then $T-I$ is nilpotent, i.e. there exists an integer $m$ such that $(T-I)^m=0$. The transformation $T-I$ is contained in the Lie algebra $\mathfrak I (\V, B)$ of $I(\V, B)$. Since the characteristic of $\F$ is large, the  Jacobson-Morozov lemma is valid for unipotent isometries. Let $\S L(2, \F)$ denote the group of all invertible $2 \times 2$ matrices over $\F$ with determinant $1$. Let $\mathfrak{sl}(2, \F)$ denote the algebra of all $2 \times 2$ matrices over $\F$ with trace zero.   The Jacobson-Morozov lemma implies that there exists a subalgebra of $\mathfrak I(\V, B)$ which contains $T-I$ and is isomorphic to $\mathfrak{sl}(2, \F)$. The corresponding algebraic group of which $\mathfrak{sl}(2, \F)$ is a Lie algebra, is $\S L(2, \F)$ or $P \S L(2, \F)=\S L(2, \F)/\{\pm I\}$, and it contains $T$. So, $T$ can be embedded in a subgroup $\pi$ of $I(\V, B)$ where $\pi$ is locally isomorphic to $\S L(2, \F)$.

Let $S: \V \to \V$ be an invertible linear transformation. An $S$-invariant subspace is said to be \emph{ indecomposable} with respect to $S$, or simply $S$-\emph{indecomposable}  if it can not be expressed as a direct sum of two proper $S$-invariant subspaces. The elementary divisors give the \emph{primary decomposition} of $\V$ into a direct sum of $S$-indecomposable subspaces and the decomposition is unique up to ``dynamical equivalence"(cf. Kulkarni \cite{kulkarni}). Each $S$-indecomposable summand in the decomposition is isomorphic to a cyclic algebra $\F[x]/((p(x)^k)$, where $p(x)$ is a prime factor of the minimal polynomial of $T$. The prime power $p(x)^k$ is an \emph{elementary divisor} of $T$. Let $\G(\V)$ denote the group of all invertible linear transformations from $\V$ onto $\V$. Suppose two elements $S$ and $T$ have the same set of elementary divisors. Then the primary decompositions of $\V$ with respect to $S$ and $T$ are isomorphic, i.e. to each summand $\V_i^S$ in the $S$-primary decomposition, there is a summand $\V_j^T$ in the $T$-primary decomposition such that $\V_i^S$ and $\V_j^T$ are isomorphic. Let $f:\V \to V$ be a linear isomorphism which maps each $\V_i^S$ onto the corresponding summand $\V_j^T$. The isomorphism $f$ conjugates $S$ and $T$. Conversely, if $S$ and $T$ are conjugates, then they have the same set of elementary divisors. Hence the elementary divisors are conjugacy invariants for $\G(\V)$, cf. Roman \cite[Theorem 7.10, p-149]{roman}, Jacobson \cite[Exercise 2, p-98]{j} for more details, and for a modern viewpoint cf. Kulkarni \cite[p-5]{kulkarni}. It turns out that the elementary divisors are also complete invariants for the conjugacy classes in $I(\V, B)$.
\begin{theorem}
Two isometries are conjugate in $I(\V, B)$ if and only if they are conjugate in $\G(\V)$.
\end{theorem}
The following is an equivalent version of this theorem.
\begin{theorem}\label{ccac}
Two isometries are conjugate in $I(\V, B)$ if and only if they have the same elementary divisors.
\end{theorem}
This is a very well-known result. There have been several proofs of this theorem, for eg. cf. \cite{kap}.
The theorem also follows from more general results like the conjugacy classification in the orthogonal and the symplectic groups over an arbitrary field of characteristic different from two cf. Milnor \cite{milnor}, Springer-Steinberg \cite{ss}, Wall \cite{wall}, Williamson \cite{will}, the conjugacy theorems in algebraic groups over an algebraically closed field cf. Seitz  \cite{seitz}, or from the categorical description of $\lambda$-hermitian forms cf. Scharlau \cite{sch} p-278.

Though the Jacobson-Morozov lemma is very useful in representation theory, none of the existing proofs of \thmref{ccac} explicitly used the Jacobson-Morozov lemma. In this note we prove \thmref{ccac} using the Jacobson-Morozov lemma. This yields a very simple proof.
A non-degenerate subspace of $(\V, B)$ is said to be  \emph{orthogonally indecomposable} with respect to an isometry $T$ if it is not an orthogonal sum of proper $T$-invariant subspaces.
Another major advantage of the use of the Jacobson-Morozov lemma is that it also classifies the orthogonally indecomposable subspaces with respect to a unipotent isometry cf. \lemref{uind} below.

\section{Preliminary results}\label{prel}
\subsection{Self-duality of the characteristic polynomial}
Let $T$ be in $I(\V, B)$. Let $\V_{ \lambda}$ denote the generalized eigenspace of $T$ with eigenvalue $\lambda$, i.e.
$$\V_{\lambda}=\{v \in \V| (T-\lambda I)^{n+1}v=0\}.$$
Then it is the (usual) eigenspace of $T_{s}$.  We have for $v, w \in \V_{ \lambda}$
$$B(v, w)= B(Tv, Tw)= B(T_sv, T_sw)= \lambda^2 B(v, w).$$
So if $B(v, w)  \not= 0$, then $\lambda = \pm 1$.
 Or to put it another way, if $\lambda \not= \pm 1$
then $B|_{\V_{\lambda}} = 0.$

 Also for $v \in \V_{ \lambda}$ and $ w \in \V_{ \mu}$ we have
$$B(v, w)= B(Tv, Tw)= B(T_sv, T_sw) = \lambda \mu B(v, w).$$

 So unless $\lambda \mu =1$ we have $\V_{\lambda}$ and
$\V_{\mu}$ are orthogonal with respect to $B$.
 Let $\oplus$ denote the orthogonal direct sum,
and $+$ the usual direct sum of subspaces. We have
\begin{equation}\label{decom}
\V = \V_{ 1} \oplus \V_{ -1} \bigoplus
\oplus_{\lambda \not= \pm 1}(\V_{\lambda} +
\V_{\lambda^{-1}}).
\end{equation}

 Moreover $B$ is non-degenerate on each  component of the
above orthogonal direct sum. That is,  $B$ induces a non-degenerate pairing   $\beta_{\lambda}: \V_{ \lambda} \times \V_{ \lambda^{-1}} \to \F$. In particular,
$dim \; \V_{\lambda} = dim \; \V_{\lambda^{-1}}$, and $B|_{\V_{\lambda}}=0=B|_{\V_{\lambda^{-1}}}$. A non-degenerate subspace of the form $$(\V_{\lambda} + \V_{\lambda^{-1}}, \;B|_{\V_{\lambda}}=0=B|_{\V_{\lambda^{-1}}})$$ is called a \emph{standard subspace}.  It follows that if $\lambda\neq \pm 1$ is an
eigenvalue of $T$, then $\lambda^{-1}$ is also an eigenvalue with the same multiplicity. Thus if $\chi_T(x)$ is the characteristic polynomial of $T$, then we have
$$\chi_T(x)=(x-1)^l (x+1)^m \chi_{oT}(x),$$
where $l, m \geq 0$ and $\chi_{oT}(x)$ is \emph{self-dual}, i.e. if $\lambda$ in $\F$ is a root, then $\lambda^{-1}$ is also a root and with the same multiplicity as $\lambda$.

The decomposition \eqnref{decom} is called the \emph{primary decomposition} of $(\V, B)$ with respect to $T$, and each non-degenerate $T$-invariant summand in the decomposition is called a \emph{primary component} of $\V$ with respect to $T$. It is clear that the conjugacy class of $T$ is determined by the conjugacy class of the restriction of $T$ on each of the primary components.

\subsection{Basic representation theory}\label{rp}
Suppose we are given two bilinear forms $B_1$, $B_2$ on two vector spaces $\U$ and $\W$. Then we can construct a bilinear form $B_1 \otimes B_2$ on $\U \otimes \W$ which is
$$(B_1 \otimes B_2)(u_1 \otimes w_1, u_2 \otimes w_2)=B_1(u_1, u_2) B_2(w_1, w_2).$$
When a vector space $\U$ is given with a non-degenerate bilinear form $B$, then we can construct a non-degenerate bilinear form $B^{{\otimes}d}$ on the $m$-th tensor product $\otimes^m \U$ using the above procedure. The form $B^{{\otimes}m}$ induces a non-degenerate bilinear form on
the $m$-th symmetric product $Sym^m(\U)$  of $\U$.

\medskip Recall that, a group representation $\pi$ on a vector space $\V$ is called \emph{irreducible} (or \emph{simple}), if it has no proper invariant subspace, i.e. the only  $\pi$-invariant subspaces are $0$ and $\V$.  It is a basic result in representation theory that
the  finite-dimensional irreducible representations of $\S L(2, \F)$ are given by symmetric products of $\F^2$ cf. Bourbaki \cite[Chapter-VIII, section-3,4]{bour}. There is a canonical symplectic form $\B_o$ on $\F^2$:
$$\hbox{for }v, w \hbox{ in } \F^2, \;\B_o(v,w)=\hbox{ the determinant of the matrix } \begin{pmatrix}v & w\end{pmatrix},$$
here we have considered the elements of $\F^2$ as column vectors.
Clearly $I(\F^2, \B_o)=\S L(2, \F)$. We identify $\F^2$ with its dual $\F^{2^{\ast}}$. Then a
a basis of $\F^2$ is given by two variables $x$, $y$, where $x$, $y$ represent the dual basis of $\F^2$. Thus for elements $u=ax + by$, $v=cx +dy$ in $\F^2$, we have $\B_o(u, v)=det \ \begin{pmatrix} a & b \\ c & d \end{pmatrix}=ad-bc$.
Then dimension of $Sym^m(\F^2)$ is $m+1$, and a basis of $Sym^m(\F^2)$ is given by
$$\{x^m, x^{m-1} y, x^{m-2}y^2,...,x^k y^{m-k}, ..., x^2 y^{m-2}, xy^{m-1}, y^m \}.$$
The symplectic form $\B_o$ induces non-degenerate $\S L(2, \F)$-invariant bilinear form $\B_m$ on $Sym^m(\F^2)$. Now note that for $u=ax+by$, $v=cx+dy$ in $\F^2$,
\begin{equation}
\B_m(\otimes^m (ax+by), \otimes^m (cx+dy))=\B_o(u, v)^m = (ad-bc)^m. \end{equation}
Expanding both sides and comparing co-efficients of the monomials $x^i y^{m-j}$ it follows that
$\B_m(x^i y^{m-i}, x^j y^{m-j})\neq 0$ if and only if $i + j =m$. Further we have
$$\B_m(x^i y^{m-i}, x^{m-i}y^i)=(-1)^i \ \frac{i! (m-i)!}{m!}.$$
This shows that $\B_m$ is symmetric, resp. symplectic if and only if $m$ is even, resp. odd.

Identifying $\V$ with $Sym^n(\F^2)$ we see that for dimension of $\V$ odd, resp. even, there is a canonical $\S L(2, \F)$-invariant symmetric, resp. symplectic bilinear form on $\V$. On a $k+1$ dimensional irreducible $\S L(2, \F)$-representation, an $\S L(2, \F)$-invariant non-degenerate bilinear form  is unique up to a constant multiple, and hence it must be $c \B_k$ for some scalar $c$.

\subsection{The Jacobson-Morozov Lemma} The Jacobson-Morozov lemma was first stated by Morozov \cite{morozov},  but his proof was incomplete.  Jacobson \cite{j2} streamlined and completed the proof. The statement of Jacobson-Morozov was for nilpotent elements in complex semisimple Lie algebras. Later it was extended by Bruhat \cite{br} to semisimple Lie algebras over fields of large characteristics. There are many versions of the Jacobson-Morozov lemma. In this exposition it is enough for us to note the following group theoretic version. We state it, as it is, in Kim-Shahidi \cite[p-405]{ks}.

\begin{theorem} (The Jacobson-Morozov Lemma)
 Suppose $u$ is a unipotent element in a semisimple algebraic group $G$. Then there exists a homomorphism $\phi:\S L_2 (\F) \to G$  such that $\phi \bigg( \begin{pmatrix} 1 & 1 \\ 0 & 1 \end{pmatrix}\bigg)=u$.
\end{theorem}

\subsection{Restriction of the form on an indecomposable subspace}\label{unind}
\begin{lemma}\label{uind}
Let $B$ is symmetric, resp. symplectic. Let $T$ be a unipotent isometry. Let $\W$ be an indecomposable subspace with respect to $T$.

(i) Then the restricted form $B|_{\W}$ is either zero, or non-degenerate.

(ii) The form $B|_{\W}$ is non-degenerate if and only if the dimension of $\W$ is odd, resp. even.
\end{lemma}
\begin{proof}
By the Jacobson-Morozov lemma, $T$ is contained in a subgroup $\pi$ of $I(\V, B)$ such that $\pi$ is locally isomorphic to $\S L(2, \F)$.

(i) Let $rad(\W)$ denote the radical of $B|_{\W}$, i.e.
$$rad(\W)=\{w \in \W\;|\;B(w, x)=0 \hbox{ for all } x\in \W\}.$$
Then $rad(\W)$ is a $\pi$-invariant subspace. We claim that $\W$ is $\pi$-irreducible.
 For otherwise $\W$ can be expressed as a direct sum of $\pi$-invariant, $\pi$-irreducible subspaces. Since $T$ is in $\pi$, this gives a decomposition of $\W$ into a direct sum of $T$-invariant subspaces. This contradicts that $\W$ is $T$-indecomposable. Hence $\W$ must be irreducible with respect to $\pi$.  Hence $rad(\W)$ is either $\W$, or $0$. This implies that $B|_{\W}$ is either $0$, or non-degenerate.

(ii) Let the dimension of $\W$ be $k+1$. Since on an irreducible $\S L(2, \F)$-representation, there is a unique, up to a constant multiple, non-degenerate $\S L(2, \F)$-invariant bilinear form, the induced $\pi$-invariant non-degenerate form on $\W$ must be $c \B_k$, for some scalar $c$. Hence dimension of $\W$ is odd, resp. even if and only if $B|_{\W}$ is non-degenerate symmetric, resp. symplectic form.

This completes the proof of the lemma.
\end{proof}

\section{Proof of \thmref{ccac}}\label{pf}
Clearly if two isometries are conjugate, they have the same elementary divisors. In the following we prove the converse.

Let $T: \V \to \V$ be an isometry. Let $\V_{\lambda}$ denote the generalized eigenspace of $T$ with eigenvalue $\lambda$. Since the elementary divisors determine the decomposition \eqnref{decom}, it is sufficient to prove the theorem on each of the primary components. So without loss of generality, we may assume that $\V$ is a primary component.

{\it Case-1}.  Let $\V=\V_{\lambda} + \V_{\lambda^{-1}}$,
$B|_{\V_{\lambda}}=0=B|_{\V_{\lambda^{-1}}}$.

Since $B$ is non-degenerate, we can choose a basis $\{e_1,....,e_m, f_1,...,f_m\}$ such that for all $i$, $e_i \in \V_{\lambda}$, $f_i \in \V_{\lambda^{-1}}$, and
$$B(e_i, e_i)=0=B(f_i, f_i),\; B(e_i, f_j)=\delta_{ij} \hbox{ or }-\delta_{ij}.$$
For each $w^{\ast} \in \V_{\lambda^{-1}}$, define the linear map
$w^{\ast}: v \to B(v, w)$. These maps enable us to identify $\V_{\lambda^{-1}}$ with the dual of $\V_{\lambda}$. Thus $T=T_L + T_L^{\ast}$, where $T_L$, the restriction of $T$ to $\V_{\lambda}$, is an element of $\G(\V_{\lambda})$.

Now suppose $T: \V_{\lambda} \to \V_{\lambda}$ is an invertible linear map and let $T^{\ast}: \V_{\lambda^{-1}} \to \V_{\lambda^{-1}}$ be its dual. Define the linear map $h_T: \V \to \V$ as follows
$$h_T(v)=\left \{ \begin{array}{ll}
T^{-1}(v) & \hbox{ if } v \in \V_{\lambda}\\
{T^{\ast}}(v) & \hbox{if } v \in \V_{\lambda^{-1}}
\end{array} \right.$$
Now observe that for $u, w \in \V_{\lambda}$,
$$B(h_Tu,h_Tw^{\ast})=h_Tw^{\ast}(h_Tu)=(T^{\ast}w^{\ast})(T^{-1}u)=w^{\ast}(TT^{-1}u)=w^{\ast}(u)=B(u,w^{\ast}).$$
This shows that $h_T$ is an isometry.

Thus in this case the conjugacy classes can be parametrized by the usual theory of linear maps. Hence the conjugacy classes are classified by the elementary divisors of an isometry.

{\it Case 2. } Suppose $T$ is unipotent.  Without loss of generality, assume $\V=\V_1$. Using \lemref{uind}, it follows that  $\V$ has a $T$-invariant orthogonal decomposition
\begin{equation}\label{v1}
\V_1=\oplus_{i=1}^{k_1} \U_i \bigoplus \oplus_{j=1}^{k_2} (\W_j + \W_j'),
\end{equation}
where for $i=1,2,...,k_1$, $\U_i$ is indecomposable with respect to $T$, and for $j=1,2,..,k_2$, $\W_j + \W_j'$ is a standard subspace. Thus the conjugacy class is determined by the restriction of $T$ on each of the components in the above orthogonal sum. So without loss of generality we may further assume that $\V$ is either indecomposable with respect to $T$, or is a standard space. If $\V$ is a standard space, there is nothing to prove, cf. case-1 above. So we may assume without loss of generality that $\V$ is indecomposable with respect to $T$.

Let $(\V', B')$ be another non-degenerate space such that dim
$\V=$ dim $\V'$. The form $B'$ is symmetric, resp. symplectic
according as $B$ is symmetric, resp. symplectic.  Let  $S: \V' \to
\V'$ be an isometry such that the elementary divisors of $S$ and
$T$ are the same. Further suppose $\V'$ is indecomposable with
respect to $S$. Clearly there is a linear isomorphism $f: \V \to
\V'$ such that $S=fTf^{-1}$. Let $B_1$ be the induced form on
$\V'$ by $f$. Since $S$ is unipotent, by the Jacobson-Morozov
lemma, there is an embedding of $S$ into a subgroup $\pi$ of
$I(\V', B')$, where $\pi$ is locally isomorphic to $\S L(2, \F)$.
Since $\V'$ is $S$-indecomposable, it must be irreducible with
respect to $\pi$. Thus there is a unique, up to a constant
multiple, non-degenerate $\pi$-invariant form. Hence $B_1=cB'$ for
some $c$ in $\F$. The transformation $C={\sqrt c}f: \V \to \V'$ is
an isometry, and $S=CTC^{-1}$.

Thus if two unipotent isometries $S$ and $T$ in $I(\V, B)$ have the same elementary divisors, then the decompositions \eqnref{v1} corresponding to the isometries are isomorphic. From the above it follows that the restriction of $S$ and $T$ on each of the isomorphic non-degenerate summands are conjugate, and hence $S$ is conjugate to $T$ in $I(\V, B)$. Thus the elementary divisors determine the conjugacy classes of unipotent isometries.

{\it Case-3. } $m_T(x)=(x+1)^d$.  Note that $-T$ is also an isometry of $\V$. Also if $m_T(x)=(x+1)^d$, then $m_{-T}(x)=(x-1)^d$ and vice-versa. Further  two isometries $S$ and $T$ are conjugate to each-other if and only if $-S$ and $-T$ are conjugates. Thus this case is reduced to the unipotent case, and the classification of conjugacy class of $T$ is similar to that of $-T$.

This completes the proof.

\subsection*{Acknowledgement}
{\it Already around 2006, Ravi Kulkarni had mentioned to me the
relevance of the Jaconson Morosov lemma   for the conjugacy
problem in the orthogonal and symplectic case.  I thank Ravi
Kulkarni for many  discussions and sharing his insights, cf
\cite{gk}. Thanks also to Dipendra Prasad for his discussions of
the Jacobson-Morozov lemma in large characteristics. Finally it is
a great pleasure to thank the referee for carefully reading this
paper and for comments.}

\end{document}